\documentclass[12pt, reqno]{amsart}
\usepackage{amsmath}%
\usepackage{amssymb}%
\usepackage[english]{babel}
\usepackage[utf8]{inputenc}
\usepackage{hyperref}
\usepackage{amsfonts}
\usepackage{setspace}
\usepackage{enumerate}
\usepackage{graphicx}
\usepackage{xcolor}
\usepackage[margin=1in]{geometry}
\usepackage[T1]{fontenc}
\usepackage{fourier}
\usepackage{mathtools}
\usepackage{appendix}
\usepackage[normalem]{ulem}

    \numberwithin{equation}{section}

\newcommand{\R}{\mathbb{R}}

\newcommand{\N}{\mathbb{N}}
\newcommand{\Q}{\mathbb{Q}}
\newcommand{\C}{\mathbb{C}}

\newcommand{\bR}{\mathbb{R}}
\newcommand{\bS}{\mathbb{S}}
\newcommand{\e}{\varepsilon}
\newcommand{\E}{\mathcal{E}}
\newcommand{\B}{\mathcal{B}}

\newcommand{\Hi}{\mathcal{H}}

\newcommand{\M}{\mathcal{M}}

\newcommand{\Conv}{{\rm conv\,}}

\newcommand{\ch}{\text{conv}}

\newcommand{\Span}{\text{span}\,}

\newcommand{\diag}{\text{diag}}

\DeclarePairedDelimiterX{\inp}[2]{\langle}{\rangle}{#1, #2}

\newtheorem{theorem}{Theorem}

 \newtheorem{lemma}[theorem]{Lemma}
 \newtheorem{proposition}[theorem]{Proposition}
 \theoremstyle{definition}
 \newtheorem{definition}[theorem]{Definition}

 \newtheorem{example}[theorem]{Example}
\counterwithout{equation}{section}

\makeatletter
\newtheorem*{rep@theorem}{\rep@title}
\newcommand{\newreptheorem}[2]{
\newenvironment{rep#1}[1]{%
\def\rep@title{#2 \ref{##1}}%
\begin{rep@theorem}}%
{\end{rep@theorem}}}
\makeatother

\newreptheorem{corollary}{Corollary}

\begin{document}

\thanks{The first and second authors were partially supported by FCT through CAMGSD, project UID/MAT/04459/2019. The third author was partially supported by FCT through CMA-UBI, project UIDB/00212/2020. The fourth author was partially supported by FCT through CIMA, project UIDB/04674/2020.}

\title{A note on the essential numerical range of block diagonal operators}

\author[L. Carvalho]{Lu\'{\i}s Carvalho}
\address{Lu\'{\i}s Carvalho, ISCTE - Lisbon University Institute\\    Av. das For\c{c}as Armadas\\     1649-026, Lisbon\\   Portugal, and Center for Mathematical Analysis, Geometry,
and Dynamical Systems\\ Mathematics Department,
Instituto Superior T\'ecnico, Universidade de Lisboa\\  Av. Rovisco Pais, 1049-001 Lisboa,  Portugal}
\email{luis.carvalho@iscte-iul.pt}
\author[C. Diogo]{Cristina Diogo}
\address{Cristina Diogo, ISCTE - Lisbon University Institute\\    Av. das For\c{c}as Armadas\\     1649-026, Lisbon\\   Portugal, and Center for Mathematical Analysis, Geometry,
and Dynamical Systems\\ Mathematics Department,
Instituto Superior T\'ecnico, Universidade de Lisboa\\  Av. Rovisco Pais, 1049-001 Lisboa,  Portugal
}
\email{cristina.diogo@iscte-iul.pt}
\author[S. Mendes]{S\'{e}rgio Mendes}
\address{S\'{e}rgio Mendes, ISCTE - Lisbon University Institute\\    Av. das For\c{c}as Armadas\\     1649-026, Lisbon\\   Portugal\\ and Centro de Matem\'{a}tica e Aplica\c{c}\~{o}es \\ Universidade da Beira Interior \\ Rua Marqu\^{e}s d'\'{A}vila e Bolama \\ 6201-001, Covilh\~{a}}
\email{sergio.mendes@iscte-iul.pt}
\author[H. Soares]{Helena Soares}
\address{Helena Soares, ISCTE - Lisbon University Institute\\    Av. das For\c{c}as Armadas\\     1649-026, Lisbon\\   Portugal\\ and Centro de Investigação em Matemática e Aplicações \\ Universidade de Évora \\  Colégio Luís António Verney, Rua Romão Ramalho, 59, 7000-671 Évora, Portugal}
\email{helena.soares@iscte-iul.pt}

\date{\today}


\begin{abstract}
In this note we characterize the essential numerical range of a block diagonal  o\-pe\-ra\-tor $T=\bigoplus_i T_i$ in terms of the numerical ranges $\{W(T_i)\}_i$ of its components. Specifically, the essential numerical range of $T$ is the convex hull of the limit superior of $\{W(T_i)\}_i$. This characterization can be simplified further. In fact, we prove the existence of a decomposition of $T$ for which the convex hull is not required.
\end{abstract}

\maketitle
\section*{Introduction}
\onehalfspacing

Block diagonal operators are a generalization of diagonal operators, hence it is not surprising that many properties of the former are emulated from the latter. An example is the numerical range $W(T_\oplus)=\Conv\left(\{W(T_n)\}_n\right)$ of a block diagonal operator $T_\oplus=\bigoplus_n T_n$, which can be seen as generalizing the formula of the numerical range of a diagonal operator $D=\diag \left((\lambda_n)_n\right)$ given by $W(D)=\Conv\left(\{\lambda_n\}_n\right)$. That is,
\[
W(D) =\Conv\big(\{\lambda_n\}_n\big)\,\,\textrm{ and }\,\,
W(T_\oplus)=\Conv\Big(\big\{W(T_n)\big\}_n\Big).
\]
In the same vein, one can ask whether $W_e(T_\oplus)$, the essential numerical range of a block diagonal operator $T_\oplus$, maintains the analogy with $W_e(D)$, the essential numerical range of a diagonal operator $D$.

Furthermore, one can also ask, since block diagonal operators are defined with respect to a decomposition,  which of such properties are independent from the choice of the decomposition. 

This note is about the pursuit of the analogy referred above and provides two descriptions of the essential numerical range: one that is independent of the choice of the decomposition (see Theorem \ref{BlocksThm}) and another one, simpler, that it is not (Proposition \ref{proposition}).



To compute the essential numerical range when $T_\oplus=\bigoplus_n T_n$ is a block diagonal operator, we are led to consider the limit superior of a sequence of sets. However, a slight adaptation is required since the essential numerical range is a convex and closed set. Thus, the main result of the paper is Theorem \ref{BlocksThm}, which states that for any decomposition the essential numerical range of a block diagonal operator $T_\oplus=\bigoplus_n T_n$ is given by
\begin{equation}\label{result}
W_e\left(T_\oplus\right)=\Conv\left(\bigcap_{k\geq 1}\overline{\bigcup_{n \geq k} W(T_n)}\right)=
\Conv\left( \overline \lim W(T_n) \right),
\end{equation}
where we use the (closed) limit superior of a sequence of sets $X_n$, $n\in\N$, as $\overline \lim X_n=\bigcap_k \overline{\bigcup_{n \geq k} X_n}$. On the other hand, the essential numerical range of a diagonal operator $D=\diag\left((\lambda_n)_n\right)$ is the closed and convex set generated by $\overline\lim\{\lambda_n\}$, the set of the limit points of the diagonal, i.e. $W_e(D)=\Conv\left(\overline\lim \{\lambda_n\}\right)$.

In this way, we recover the analogy from the numerical range of a diagonal and a block diagonal operator. In particular, the essential numerical range of a diagonal operator is a special instance of the block diagonal case. We deduce from (\ref{result}) that
\[
W_e(D) =\Conv\left(\overline\lim \{\lambda_n\}\right)\,\,\textrm{ and }\,\,
W_e(T_\oplus)=\Conv\left( \overline \lim W(T_n) \right).
\]

Regarding the existence of a decomposition $T_\oplus=\bigoplus_n \widetilde{T}_n$ where the convex hull is not required, we prove:
\[
W_e(T)= \bigcap_{k\geq 1}\overline{\bigcup_{n \geq k} W(\tilde{T}_n)}.
\]

We remark that throughout the text, when we refer to limit superior we always mean the closed limit superior as above and not the usual definition.



We now recall some terminology and fix the notation used throughout the text.

Let $\mathcal{H}$ be a complex separable  Hilbert space with inner product $\langle \cdot\,,\cdot\,\rangle$.  Denote $\bS_\Hi=\{x\in\Hi:\inp{x}{x}=\|x\|^2=1\}$ the unit sphere in $\Hi$. 
As usual, $\mathcal{B}(\Hi)$ denotes the algebra of bounded linear operators on $\Hi$ and $\mathcal{K}(\Hi)$ its closed ideal of compact operators.

The numerical range of an operator $T\in\mathcal{B}(\mathcal{H})$ is the set
\[
W(T)=\{\langle Tx,x\rangle:x\in \bS_\Hi\}\subset\C,
\]
and the essential numerical range of $T$ is defined by
\[
W_{e}(T)=\bigcap_{K\in\mathcal{K}(\mathcal{H})}\overline{W(T+K)}.
\]
It is well known that $W_e(T)\subseteq \overline{W(T)}$, see \cite{FSW}. 

Now we introduce the notion of essential sequence. As usual, we write $x_n \rightharpoonup x$ if a sequence $(x_n)_n$ in $\Hi$ converges to $x\in \Hi$ in the weak topology.

\begin{definition}
   An essential sequence $(x_n)_n\subset \Hi$ for $\omega \in\C$ is a sequence of unit vectors such that $x_n\rightharpoonup 0$  and $\langle Tx_n, x_n\rangle\rightarrow q$.
\end{definition}
The existence of an essential sequence for an element $\omega\in \C$ is a  necessary and sufficient condition for $\omega$ be in $W_e(T)$ (see \cite[Corollary in page 189]{FSW}).  

We will be especially dealing with convex sets, so it is useful to introduce the simplex of the space of sequences $\ell^2(\bR)$, that is $\Delta_{\ell^2(\bR)}=\{(a_n)_n\in \ell^2(\bR): \sum_{n\geq 1} a_n=1, a_n\geq0 \}$. Furthermore, the notion of extreme point will also be important. A point $z\in W(T)$ is called an extreme point if, given a convex combination $z=\alpha z_1+(1-\alpha)z_2$, with $\alpha\in(0,1)$ and $z_1,z_2\in W(T)$, then necessarily $z=z_1=z_2$.


Although it is folklore knowledge, we introduce a detailed account on block diagonal operators (which is usually overlooked) to facilitate the proofs.
Set $(t_n)_n$ to be an increasing sequence of non-negative integers, with $t_0=0$, and $\ell_n=t_n-t_{n-1}$ for positive integers. Fix orthonormal bases $\mathcal{U}_n=\{u_i^{(n)} \,:\, 1\leq i\leq \ell_n\}$ for $\C^{\ell_n}$, for all $n\in\N$. Let $\left(A_n\right)_n$ be a sequence of matrices, where each matrix $A_n\in\M_{\ell_n}(\C)$ is defined with respect to the basis $\mathcal{U}_n \in \C^{\ell_n}$ so that its entries are given by
\[
a_{ij}^{(n)}= \left\langle A_nu_j^{(n)},u_i^{(n)}\right\rangle,
\]
for $1\leq i,j\leq \ell_n$.

Fix a Hilbert basis $\mathcal{B}=\{e_n:n\in\N\}$ for $\Hi$. The elements of $\mathcal{B}$ may be assembled as
\[
\mathcal{B}=\{\underbrace{e_1,\ldots,e_{t_1}}_{\mathcal{B}_1},\underbrace{e_{t_1+1},\ldots,e_{t_2}}_{\mathcal{B}_2},\ldots,
\underbrace{e_{t_{n-1}+1},\ldots,e_{t_n}}_{\mathcal{B}_n},\ldots\}.
\]
That is, $\mathcal{B}=\bigcup_{n\geq 1} \mathcal{B}_n$, where $\mathcal{B}_n=\{e^{(n)}_i=e_{t_{n-1}+i}:1\leq i\leq \ell_n\}$. According to this decomposition, $\Hi$ is the direct  sum of finite dimensional orthogonal Hilbert (sub)spaces  $\Hi_n := \Span {\mathcal{B}_n}$ of $\Hi$ and we have $\Hi=\bigoplus_{n\geq 1}\Hi_n$. In particular, $x\in\Hi$ if and only if $x=(x_n)_{n}$, with $x_n\in\Hi_n$ and $\sum_{n\geq 1}\lVert x_n\rVert^2 < \infty$. The inner product is $\langle x,y \rangle=\sum_{n\geq 1}\left\langle x_n,y_n \right\rangle_{\Hi_n}$, where $x=(x_n)_n$, $y=(y_n)_n$ with $x_n, y_n \in  \Hi_n$.


%


Taking into account this decomposition, a block diagonal bounded operator $T:\Hi \to \Hi$ (with respect to the decomposition $\Hi=\bigoplus_{n\geq 1}\Hi_n$) can be defined as follows. 
For each $n\in\N$, let $T_n:\Hi_n\to\Hi_n$ be the bounded operator given by
\[
T_n\left(e^{(n)}_i\right)=\left(\phi_n A_n \phi_n^{-1}\right)\left(e^{(n)}_i\right), \,\,\forall\, i=1,\ldots,\ell_n,
\]
where $\phi_n: \C^{\ell_n} \to \Hi_n$ is an isometry of Hilbert spaces defined by 
\[
\phi_n\left(u_i^{(n)}\right)=e_i^{(n)}.
\]
Now define 
\begin{align*}
T=\bigoplus&_{n\geq 1}T_n: \Hi\to \Hi\\
x&\mapsto T(x)=T((x_n)_n)=\left(T_n(x_n)\right)_n
\end{align*}
In order to $T$ to be bounded it is necessary and sufficient that $\lVert T \rVert=\sup_{n}\lVert T_n\rVert <\infty$. Usually, by abuse of notation, we identify the operator $T_n$ with the matrix $A_n$. We thus write $T=\bigoplus A_n$ and we have 
\[
\inp{Tx}{x}=\sum_n \inp{Tx_n}{x_n}=\sum_n \inp{T_nx_n}{x_n}=\sum_n \inp{A_n\phi_n^{-1}x_n}{\phi_n^{-1}x_n}.
\]
The construction of a block diagonal operator $T=\bigoplus_n A_n$ depends on the fixed decomposition $\Hi=\bigoplus \Hi_n$ (and on the operators $A_n: \Hi_n \to \Hi_n$). Obviously, this decomposition may not be unique since it may exist some other decomposition, $\Hi=\bigoplus_n \tilde \Hi_n$, giving rise to a new block diagonal decomposition $T=\bigoplus_n \tilde A_n$, where $\tilde A_n: \tilde\Hi_n \to \tilde\Hi_n$. A natural question is for which decompositions of $T$ is the computation of the essential numerical range simpler. Proposition \ref{proposition} provides a partial answer to this question for it provides, as already mentioned, a decomposition for which the convex hull in equation (\ref{result}) can be discarded.




\section*{Essential numerical range}

Our main result states that the essential numerical range of $T$ is the convex hull of the set $\bigcap_{k=1}^\infty \overline{\bigcup_{n \geq k} W(A_n)}$, i.e. it is the (closed) convex hull of the limit superior of the sets $W(A_n)$, for all $n$. To prove it, we need the following Lemma.
\begin{lemma}
If $S_k \subset \C$ be a decreasing sequence of compact sets, then
\[
 \bigcap_{{k\geq 1}} \mathrm{conv}\big(S_k\big)= \mathrm{conv}\big(\bigcap_{k\geq 1} S_k \big).
\]
\end{lemma}
\begin{proof}


Since $\bigcap_{k\geq 1} S_k \subseteq S_j$,  it follows that $\mathrm{conv}\big(\bigcap_{k\geq 1} S_k \big)\subseteq\Conv\big(S_j\big)$, for any $j$. It follows that $\mathrm{conv}\big(\bigcap_{k\geq 1} S_k \big)\subseteq\bigcap_{j\geq 1} \Conv\big(S_j\big)$.

For the reverse inclusion we use Carathéodory's theorem. More precisely, since $\C\cong\bR^2$, any $s \in  \Conv(S_k) \subseteq \C$ is a convex combination of at most $3$ points in $S_k$.  For each $k$, denote each of these points by $s^{(i)}_k \in S_k$, $i=1, 2, 3$. So $s=\sum_{i=1}^3 \alpha_k^{(i)}s^{(i)}_k$, for some $(\alpha_k^{(1)},\alpha_k^{(2)},\alpha_k^{(3)}) \in \Delta_{\bR^3}=\{(\alpha_1,\alpha_2,\alpha_3) \in \bR^3: \sum_{p=1}^{3} \alpha_p=1, \alpha_p\geq 0\}$. Recalling that $S_k$ is a decreasing sequence of sets, $s_k^{(i)}\in S_1$, for all $k$ and $i$. Compactness of $\Delta_{\bR^3}$ and $S_1$ allows us to take a subsequence $(k_n)_n$ such that $s^{(i)}_{k_n}$ and $\alpha^{(i)}_{k_n}$ are convergent subsequences to $s^{(i)}$ and $\alpha^{(i)}$, respectively. Clearly, the limit of $\left(\alpha^{(1)}_{k_n}, \alpha^{(2)}_{k_n}, \alpha^{(3)}_{k_n}\right)$ belongs to the simplex, that is $\left(\alpha^{(1)},\alpha^{(2)},\alpha^{(3)}\right) \in \Delta_{\bR^3}$. On the other hand,   $s^{(i)}_{k_n} \in S_{k_n}\subset S_j$ for all $k_n \geq j$. Since the sets $S_j$ are closed,  the limit of $s^{(i)}_{k_n}$ belongs to $S_j$, for any $j \in \N$.
Thus $s^{(i)} \in \bigcap_{k\geq 1} S_k$ and $s=\sum_{i=1}^3 \alpha^{(i)} s^{(i)} \in \Conv\left(\bigcap_{k\geq 1} S_k\right)$.

\end{proof}

Now we can prove our main theorem.

\begin{theorem}\label{BlocksThm}
Let $T=\bigoplus_{n\in\N}A_n \, \in \B(\Hi)$. Then,
\[
W_e(T)= \mathrm{conv}\Big(\bigcap_{k\geq 1}\overline{\bigcup_{n \geq k} W(A_n)}\Big).
\]
\end{theorem}
\begin{proof}
We will prove that $W_e(T) =\bigcap_{k\geq1} \mathrm{conv}\Big(\overline{\bigcup_{n \geq k} W(A_n)}\Big)=:\mathcal{W} $. The Theorem then follows from the previous Lemma with $S_k=\overline{\bigcup_{n \geq k} W(A_n)}$.

Cantor's intersection Theorem guarantees the set $\mathcal{W}$ is non-empty. We will start by finding an essential sequence for any $\omega \in \mathcal{W}$, thus proving that $\mathcal{W} \subseteq W_e(T)$.
First, suppose that, for each $k\in\N$, $\omega \in \Conv\Big(\bigcup_{n \geq k} W(A_n)\Big)$. Then, there are $\omega_n^{(k)} \in W(A_n)$ and
$\left(\left(\alpha_n^{(k)}\right)^2\right)_{n\geq k}\in\Delta_{\ell^2(\R)}$ such that
 \[ \omega=\sum_{n\geq k} \left(\alpha_n^{(k)}\right)^2\omega_n^{(k)}.\]
Let $x_n^{(k)} \in \bS_{\C^{\ell_n}}$ be such that $\omega_n^{(k)}=\left\langle A_n x_n^{(k)}, x_n^{(k)}\right\rangle$, for all $n\geq k$. For each $k$ and $n$ we denote by $\hat{x}_{n}^{(k)}$ the counterpart in $\Hi_n$ of $x_n^{(k)} \in \bS_{\C^{\ell_n}}$. That is, according to the notation in the introduction, $\hat{x}_{n}^{(k)}=\phi_n(x_{n}^{(k)}) \in \Hi_n$.
Consider the sequence of vectors $y^{(k)}\in \Hi$ defined as follows:
\[
y^{(k)}= \sum_{n\geq k} \alpha_n^{(k)} \hat x_{n}^{(k)}, \quad k\in\N.
\]

From $\|\hat{x}_{n}^{(k)}\|=\|x_{n}^{(k)}\|=1$ and $\hat x_{n}^{(k)}  \bot \; \hat x_{m}^{(k)}$, for $n \neq m$, we may conclude that $\|y^{(k)}\|=1$:
\[
\left\lVert y^{(k)}\right\rVert^2= \sum_{n\geq k} \left(\alpha_n^{(k)}\right)^2 \left\lVert \hat x_{n}^{(k)}\right\rVert^2 = \sum_{n\geq k} \left(\alpha_n^{(k)}\right)^2 =1.
\]

Let us now see that $y^{(k)}  \underset{k}{\rightharpoonup} 0$ by proving that $\langle y^{(k)},z\rangle\to 0$, for any $z\in\Hi$. Write $z=\sum_{n\geq 1} z_n$, with $z_n \in \Hi_n$.
We have $\|z\|^2=\sum_{n\geq 1} \|z_n\|^2<\infty$. Thus, there exists $N$ such that, for $K\geq N$, the vector $z^{(K)}=\sum_{n\geq K} z_n$ has norm $\|z^{(K)}\|\leq \e$. Since $y^{(K)} \in \bigcup_{n\geq K}\Hi_n$ and $\|y^{(K)}\|=1$ we have
\begin{equation} \label{epsiloneq}
    \Big|\langle y^{(K)}, z \rangle\Big|=\Big|\langle y^{(K)}, z^{(K)} \rangle\Big|\leq \|y^{(K)}\| \|z^{(K)}\| \leq \e, 
\end{equation}
for any $z\in \Hi$. Thus $y^{(k)}  \underset{k}{\rightharpoonup} 0$.

Finally, since $T$ is the block diagonal operator $T=\bigoplus_{n\in\N}A_n$,  we have
\begin{align*}
\inp{T\hat x_n^{(k)}}{\hat x_n^{(k)}}=\inp{A_n x_n^{(k)}}{x_n^{(k)}}= \omega_n^{(k)}.
\end{align*}
 So it easily follows, using that $T \hat x_{n}^{(k)} \in \Hi_n$ and $\Hi_n\bot \Hi_m$ for $n \neq m$, that
\begin{equation} \label{omegaeq}
\langle T y^{(k)}, y^{(k)}\rangle  = \sum_{n\geq k} \left(\alpha_n^{(k)}\right)^2 \inp{T\hat x_n^{(k)}}{\hat x_n^{(k)}}  =\sum_{n\geq k} \left(\alpha_n^{(k)}\right)^2\omega_n^{(k)}= \omega.
\end{equation}
Thus $\left(y^{(k)}\right)_k$ is an essential sequence for $\omega$, that is, $\omega\in W_e(T)$.

Now suppose that, for each $k\in\N$, $\omega \in \Conv\left(\overline{\bigcup_{n \geq k} W(A_n)}\right)=\overline{\Conv\left(\bigcup_{n \geq k} W(A_n)\right)}$. So, there is a sequence $(\omega_m)_m\subset \Conv\left( \bigcup_{n \geq k} W(A_n) \right)$ converging to $\omega$. According to the previous discussion, each element $\omega_m$ has an essential sequence $\left(y^{(k)}_m\right)_k$, with unitary $y^{(k)}_m\in \bigcup_{n\geq k} \Hi_n$ and, by (\ref{omegaeq}), $\left\langle T y_m^{(k)}, y_m^{(k)}\right\rangle=\omega_m$, for any $k$. Forming a sequence from the diagonal elements $\Big(y_m^{(m)}\Big)_m$, we thus have that $\|y_m^{(m)}\|=1$ and $\langle T y_m^{(m)}, y_m^{(m)}\rangle =\omega_m$. Observing that $y_m^{(m)} \in \bigcup_{n \geq m} \Hi_n$, and using arguments similar to those that led to equation (\ref{epsiloneq}), we can conclude that $y_m^{(m)} \rightharpoonup 0$. Hence, $\left(y_m^{(m)}\right)_m$ is an essential sequence for $\omega$ and we have then proved that $\omega \in \Conv\left(\overline{\bigcup_{n \geq k} W(A_n)}\right)$. Therefore, $\omega \in \bigcap_{k\geq1} \mathrm{conv}\Big(\overline{\bigcup_{n \geq k} W(A_n)}\Big)\subset\mathcal{W}$.

To prove the converse, take $ \omega \in W_e(T)$ and one of its essential sequences $\left(y^{(m)}\right)_m$ in $\bS_\Hi$. Take arbitrary $\e>0$ and $k\in\N$. Denote by $P_{k}$ the projection onto $\bigcup_{n\leq k} \Hi_n$. The fact that $y^{(m)}  \underset{m}{\rightharpoonup} 0$ implies that the projection over any finite number of coordinates converges strongly to zero, for any $k \in \N$, hence  we must have  $ \left\| P_{k}y^{(m)}\right\| \underset{m}{\longrightarrow} 0$. Together with the hypothesis that $\left\langle Ty^{(m)}, y^{(m)} \right\rangle \to \omega$, we can pick a positive integer $N$  that simultaneously satisfies $\left\| P_{k}y^{(N)}\right\|<\e$ and $|\omega-\inp{Ty^{(N)}}{y^{(N)}}|<\e$. For simplicity, we denote $y=y^{(N)}$ and $\tilde \omega=\inp{Ty}{y}$. Thus, we can rewrite our previous conditions as
\begin{eqnarray}\label{desig_epsilon}
 \big\| P_{k}y\big\| < \e
\quad
\text{and} \quad |\omega-\tilde \omega|\leq \e.
\end{eqnarray}

The strategy to show that $\omega\in\Conv\Big(\overline{\bigcup_{n \geq k} W(A_n)}\Big)$ will be to construct a convex combination $\sum_{n \geq k+1} \alpha_n\langle A_n x_n, x_n\rangle$ of elements $\langle A_n x_n, x_n\rangle \in W(A_n)$, for  $n \geq k+1$. By proving that this convex combination is arbitrarily close to $\tilde \omega$, we prove that $\tilde \omega$ is arbitrarily close to $\Conv\Big(\bigcup_{n \geq k+1} W(A_n)\Big)$.  But from (\ref{desig_epsilon}), $\tilde \omega$ was arbitrarily close to $\omega$ and therefore this will allow us to conclude that $\omega$ is arbitrarily close to $\Conv\Big(\bigcup_{n \geq k+1} W(A_n)\Big)$. That is, $\omega \in \overline{\Conv\Big(\bigcup_{n \geq k+1} W(A_n)\Big)}=\Conv\Big(\overline{\bigcup_{n \geq k+1} W(A_n)}\Big)$. Since this is true for each $k\in\N$, we have
\[
\omega \in \bigcap_{k \geq 1} \Conv\Big(\overline{\bigcup_{n \geq k} W(A_n)}\Big)=\mathcal{W}.
\]

To construct the referred convex combination  $\sum_{n \geq k+1} \alpha_n\langle A_n x_n, x_n\rangle$, recall that $y=y^{(N)}$ and write  $y=(y_n)_n \in \Hi$, with $y_n \in \Hi_n$. Let $\hat x_n \in \Hi_n$ be the unitary vectors with the same direction of that of $y_n \in \Hi_n$, that is, $\|y_n\|\hat x_n=y_n$.
Denote $x_n=\phi_n^{-1}(\hat x_n)  \in \C^{\ell_n}$, for $n \geq k$, and $\beta= \frac{1}{\| (I- P_{k}) y\| ^2} $. Note that $1=\|y\|^2=\|P_{k} y\| ^2+\| (I- P_{k}) y\| ^2$ and recall from  (\ref{desig_epsilon}) that $\|P_{k} y\| ^2< \e^2$.   Hence, $\| (I- P_{k}) y\| ^2>1-\e^2>0$ if we pick $\e$ small enough.

Let $ \alpha_n=\frac{\| y_n\|^2}{ \|(I- P_{k}) y\| ^2}\geq 0$. Since  $\sum_{n \geq k+1}\|y_n\|^2= \|(I- P_{k}) y\| ^2$ then $\sum_{n\geq k+1} \alpha_n= 1$. We thus obtain
\[
\sum_{n \geq k+1} \alpha_n\langle A_n x_n, x_n\rangle=\sum_{n \geq k+1} \alpha_n\langle T_n \hat x_n, \hat x_n\rangle=\beta \sum_{n\geq k+1}\langle T_n y_n,  y_n \rangle=\beta\Big(\tilde\omega-\sum_{n \leq k} \left \langle T_n y_n,  y_n \right\rangle\Big).
\]
%
%
Using (\ref{desig_epsilon}) again, we have $ \| P_{k} y\|^2 = 1- \|\left(I- P_{k}\right)y\| ^2 < \varepsilon^2$ and  $\beta-1< \frac{1}{1-\varepsilon^2}-1=\frac{\e^2}{1-\e^2}$. Hence,
\begin{align*}
\Big| \sum_{n \geq k+1} \alpha_n\inp{ A_n x_n}{x_n}-\tilde \omega\Big|& = \Big| \beta \big( \tilde \omega-\sum_{n \leq k}\inp{ T_n y_n}{y_n}\big)-\tilde\omega \Big|\\
& \leq  (\beta-1) \left|\tilde \omega \right|+\beta \big| \sum_{n \leq k} \langle T_n  y_n,  y_n \rangle\big|\\
&< \frac{\e^2}{1-\e^2} \left|\tilde \omega\right| + \beta\big|\inp{T\big(P_ky\big)}{P_ky}\big| \\\
&< \frac{\e^2}{1-\e^2} |\tilde \omega| + \beta  \|T\|\e^2.
\end{align*}
We have just proved that for any $k$ the convex combination $\sum_{n \geq k+1} \alpha_n\inp{ A_n x_n}{x_n}$ is arbitrarily close to $\tilde \omega$, as intended.
\end{proof}

Here are two simple applications of the above result.

\begin{example}\label{ex_diag}
If $T=\bigoplus_{n\in\N}A_n \, \in \B(\Hi)$, with $A_n=\left[\lambda_n\right]\in\M_{1\times 1}(\C)$, for all $n\in\N$, then $T$ is a diagonal operator. Recall from our discussion in the introduction the definition of closed limit superior of a sequence of sets. Theorem \ref{BlocksThm} says that
\[
W_e(T)= \Conv\Big(\bigcap_{k\geq 1}\overline{\bigcup_{n \geq k} \left\{\lambda_n\right\} }\Big)=\Conv\left(\overline\lim \left\{\lambda_n\right\}\right).
\]
\end{example}

\begin{example}
Let $T=\bigoplus_{n\in\N}A_n \, \in \B(\Hi)$, with $A_n=A$, for all $n\in\N$. Then, applying theorem \ref{BlocksThm} and the fact that $W(A)$ is convex and closed, we find that
\[
W_e(T)= \Conv\Big(\bigcap_{k\geq 1}\overline{\bigcup_{n \geq k} W(A)}\Big)=W(A).
\]

\end{example}
Theorem \ref{BlocksThm} gives us a way to find the essential numerical range of $T=\bigoplus_n T_n$ for any block decomposition $(T_n)_n$ of $T$. As previously stated, the block operator decomposition is not unique and for a particular decomposition, judiciously chosen, the result  can be further simplified. We can find a way to decompose the operator $T$ as $T=\bigoplus_n T_n$ so that the convex hull can be dismissed in the formula in theorem \ref{BlocksThm}. The relevance of the next proposition in terms of practical calculations is reduced. In fact, it is harder to explicitly describe the decomposition constructed therein than it is to directly use the result from the preceding theorem. However, theoretically, it may be useful to have such result at hand. A formulation close to the one in our proposition is used in (\cite{A}). Nevertheless, besides containing an inaccuracy in the positioning of the closure, there is no reference to the fact that it holds true only for some decompositions of $T$. Also, the result is stated without a proof.

Next proposition proves rigorously the existence of a particular decomposition where the essential numerical range can be described without the convexification, but before that we need two preparing lemmata. The first of these lemmas  states that the proposition is invariant under translations.
\begin{lemma}\label{translation}
Let $z \in \C$ and let $T=\bigoplus T_n \in B(\Hi)$ be a block diagonal operator. 
\[
W_e(T)= \bigcap_{k\geq 1}\overline{\bigcup_{n \geq k} W(T_n)} \quad 
\text{if and only if}\quad W_e(T-zI)= \bigcap_{k\geq 1}\overline{\bigcup_{n \geq k} W(T_n-zI)}.
\]
\end{lemma}
\begin{proof}
We start by noting that if $T$ is the block diagonal operator with  decomposition $T=\bigoplus T_n$ then  $T-zI$ is also a block diagonal operator with decomposition $\bigoplus T_n-zI=\bigoplus \left(T_n-zI\right)$.

Since $W_e(T-zI)=W_e(T)-z$ and $W(T_n)=W(T_n-zI)+z$, if
\[
W_e(T)= \bigcap_{k\geq 1}\overline{\bigcup_{n \geq k} W(T_n)},\]
we have that
\[
W_e(T-zI) = W_e(T)-zI = \bigcap_{k\geq 1}\overline{\bigcup_{n \geq k} W(T_n)-zI} = 
\bigcap_{k\geq 1}\overline{\bigcup_{n \geq k} W(T_n-zI)}.
\]
The other implication follows from the previous one, considering $\tilde{T}=T+zI$.
\end{proof}

On a convex set $A \subset\C$ a point $a \in A$ is an extreme point if it is not a convex combination of any other points in the set $A$. That is, we say that $a$ is an extreme point of $A$ if when $a,b,c \in A$ and $a=\alpha b+(1-\alpha) c$, for some $\alpha \in [0,1]$, then $b=c=a$. We denote by $\E(T)$ the set of extreme points of $W_e(T)$.   Since $W_e(T)$ is compact and convex, Krein-Milman theorem asserts that $W_e(T)=\Conv\left(\E(T)\right)$.

On the other hand, for any $\omega \in \C\setminus \{0\}$ we can write uniquely $\omega=|\omega|e^{i\theta}$, for some $\theta \in [0, 2\pi)$. Thus we can define an angle function  $\theta: \C\setminus \{0\} \to [0, 2\pi)$. Denote  $\theta_{|A}$  the  restriction  of the function $\theta$ to the set $A$. Next lemma shows that we can find a translation $T-zI$ of the operator $T$ where $\theta_{|\E(T-zI)}$, the restricted angle function to the set of extreme points  $\E(T-zI)$, is injective. That is, if we apply an adequate translation there are no two  extreme points of  $\E(T-zI)$ with the same angle. 
\begin{lemma}\label{theta_injectivity}
Let $T \in B(\Hi)$. There exists $z \in \C$ such that  $\theta_{|\E(T-zI)}$ is injective.
\end{lemma}
\begin{proof}
In the case the essential numerical range of $T$ is a singleton, $W_e(T)=\{a\}=\E(T)$, there are two possibilities: either $a \neq 0$ and then $\theta_{|\E(T)}$ is trivially injective, or $a=0$, in which case  $W_e(T-zI)=\{-z\}=\E(T-zI)$, for any $z\in \C\setminus \{0\}$, and hence again $\theta_{|\E(T-zI)}$ is trivially injective.

Otherwise, if $W_e(T)$ is not a singleton, the set of extreme points $\E(T)$ has at least two elements, $\omega_1, \omega_2 \in \E(T)$. We can choose $z\in \ch\{\omega_1, \omega_2\} \subseteq W_e(T)$, with $z\not \in \{0,\omega_1, \omega_2\} $, and translate $W_e(T)$ by $z$. It is clear that $0 \in W_e(T-zI)\setminus \E(T-zI)$, since $z$ was an element in $W_e(T)$ and not in $\E(T)$. Thus the function $\theta_{|\E(T-zI)}$ is well-defined since $0$ is not in its domain. On the other hand, injectivity of $\theta_{|\E(T-zI)}$ is a consequence of $W_e(T-zI)$ being convex; if $\theta(z_1)=\theta(z_2)$, for some $z_1,z_2 \in \E(T-zI)$ with $|z_1|<|z_2|$, $z_1$ would be a convex combination of $0 \in W_e(T-zI)$ and $z_2 \in W_e(T-zI)$, thus not an extreme point of $W_e(T-zI)$.
\end{proof}
\begin{proposition}\label{proposition}
Let $T=\bigoplus_n T_n$ be a block diagonal operator. Then, there is a decomposition  for which $T=\bigoplus_n \tilde{T}_n$ and
\begin{equation}\label{We_specialdecom}
W_e(T)= \bigcap_{k\geq 1}\overline{\bigcup_{n \geq k} W(\tilde{T}_n)}.
\end{equation}
\end{proposition}
\begin{proof}
Lemma \ref{theta_injectivity} assures the existence of $z\in\C$ such that $\theta_{|\E(T-zI)}$ is injective. If we prove \eqref{We_specialdecom} for $W_e(T-zI)$, from lemma \ref{translation} this result holds true for $W_e(T)$. Therefore, we can assume without loss of generality that $T$ is such that $\theta_{|\E(T)}$ is injective. 

In order to simplify notation, write $W_n=W(T_n)$, $\tilde W_n=W(\tilde T_n)$, $W_e=W_e(T)$ and $\E=\E(T)$. First, observe that for any decomposition of $T$ we have
\begin{equation}\label{ext_in_intersection}
\E \subset \bigcap_{k\geq 1}\overline{\bigcup_{n \geq k} W_n}.
\end{equation}
This follows from Theorem \ref{BlocksThm}, which asserts that
$
W_e= \Conv\Big(\bigcap_{k\geq 1}\overline{\bigcup_{n \geq k} W_n}\Big)
$
and by definition of extreme point.

Moreover, note that  if $\omega_e$ is a point in $\overline{\bigcup_{n \geq k} W_n}$, for all $k\in\N$, then for any $\zeta>0$ there is a $\gamma_{n_k}\in W_{n_k}$, with $n_k\geq k$, such that $\left|\gamma_{n_k}-\omega_e\right|<\zeta$. Hence, an extreme point $\omega_e$ of $W_e$ has the following property:
\begin{equation}\label{infinite_seq}
\text{for any } \zeta>0 \text{ there is a sequence } (n_k)_k, \text{ with } n_k\geq k, \text{ such that }  d\big(\omega_e, W_{n_k}\big)<\zeta.
\end{equation}

Consider an arbitrary $\e>0$. We next construct a decomposition $\left(\tilde T_m\right)_m$ of the operator $T$ for which the result holds.

When $m=1$ pick one element $\omega_e \in \E$. In view of (\ref{infinite_seq}), take $M_1$ to be such that $d(\omega_e,W_{M_1})<\e$ and let $\tilde T_1=\bigoplus_{p=1}^{M_1} T_p$.

For $m \geq 2$, we start by dividing the interval $[0, 2\pi)$ into $m$ subintervals $2\pi\left[j/m, (j+1)/m\right)$, for $j=0, \ldots, m-1$. If $\theta(\E) \bigcap 2\pi\left[j/m, (j+1)/m\right)\neq\emptyset$, pick any $\omega_e \in \E$ such that $\theta(\omega_e) \in 2\pi\left[j/m, (j+1)/m\right)$. Denote such point by $\omega_j^{(m)}$. For those pairs $(m,j)$ with
\[
\theta(\E) \cap 2\pi\left[\frac{j}{m}, \frac{j+1}{m}\right)=\emptyset,
\]
we let $\omega_j^{(m)}= \omega_{j'}^{(m)}$, where $j' \in \{0, \dots, m-1\}$ such that $\theta(\E) \cap 2\pi\left[\frac{j'}{m}, \frac{j'+1}{m}\right)\neq \emptyset$.
According to (\ref{infinite_seq}), for each of these $m$ points $\omega_j^{(m)}$ in $\E$ there is an integer $m_j \geq M_{m-1}+1$ such that $d\big(\omega_j^{(m)}, W_{m_j}\big)<\e/m$. Let $M_m=\max_{0\leq j\leq m-1} m_j$ and define
\[
\tilde T_m=\bigoplus_{p=M_{m-1}+1}^{M_m} T_p.
\]
Since $M_{m}>M_{m-1}$, we have that $M_m \underset{m}{\longrightarrow} +\infty$ and it is clear that $T$ is the direct sum of the operators $\tilde T_m$. 

Claim: the decomposition $\left(\tilde T_m\right)$ verifies (\ref{We_specialdecom}).
We start by proving the inclusion $W_e\subset \bigcap_{k\geq 1}\overline{\bigcup_{n \geq k} \tilde{W}_n}$. Recall that $\tilde W_m=W(\tilde T_m)$. So, for each $m\geq 1$,
\[
\tilde W_m=W\left(\bigoplus_{p=M_{m-1}+1}^{M_m} T_p\right)=\Conv\big\{W_p:M_{m-1}+1\leq p \leq {M_m}\big\}.
\]
Since $d\left(\omega_j^{(m)}, W_{m_j}\right)<\e/m$ and $M_{m-1}<m_j\leq M_m$, for all $j=0,\ldots, m$, then
\begin{equation}\label{dist_Wm}
d\left(\omega_j^{(m)}, \tilde W_m\right)\leq d\left(\omega_j^{(m)}, W_{m_j}\right)  < \e/m.
\end{equation}


Let $\omega_e \in \E$.
For any $m$, there is $j \in \{0, \dots, m-1\}$ 
such that $\theta(\omega_e) \in 2\pi[j/m, (j+1)/m)$. Then, the corresponding $\omega_j^{(m)}$, whose angle $\theta(\omega_j^{(m)})$  belongs to the same interval, satisfies $\left|\theta(\omega_e)-\theta\left(\omega_j^{(m)}\right)\right|<  2\pi/m$. Denote such $\omega_j^{(m)}$ by $\gamma_m$. We thus obtain a sequence $(\gamma_m)_m$ such that  $\theta(\gamma_m) \to \theta(\omega_e)$. Since $\gamma_m, \omega_e \in \E$ and  $\theta(\gamma_m) \to \theta(\omega_e)$,  by injectivity of $\theta$ in $\E$, we have $ \gamma_m \to \omega_e$.
So, for any $\delta>0$, we can pick $M\in\N$ such that, for $N\geq M$,  $d(\gamma_N, \omega_e)<\delta/2$ and $d\big(\gamma_N, \tilde W_{N}\big)< \e/N< \delta/2$ (see (\ref{dist_Wm})), which implies that
\begin{equation}\label{dist_WN}
d\big(\omega_e, \tilde W_{N}\big)\leq d(\omega_e, \gamma_N)+d\big(\gamma_N, \tilde W_{N}\big)< \delta.
\end{equation}
Hence, we have $\E \subseteq \bigcap_{k \geq 1}\overline{\bigcup_{n \geq k}\tilde W_{n}}$.

We now claim that $\Conv(\E) \subset \bigcap_{k \geq 1}\overline{\bigcup_{n \geq k}\tilde{W_n}}$. If $\omega \in  \Conv(\E)$ there are $\omega_e^i \in \E$, for $i=1,2,3$, such that $\omega=\sum_i \alpha_i \omega_e^i$, for some $(\alpha_1,\alpha_2,\alpha_3) \in \Delta_{\bR^3}$. From (\ref{dist_WN}), we know that for any $\delta>0$, there is  $M_i$ such that $d(\omega_e^i, \tilde W_N)<\delta$ for $N \geq M_i$. Then, for $N \geq \max\{M_1,M_2,M_3\}$ it follows that
\[
d(\omega, \tilde W_N)=d\left( \sum_i \alpha_i \omega_e^i, \tilde W_N\right) \leq  \sum_i \alpha_i  d\left(\omega_e^i, \tilde W_N\right) < \delta,
\]
using the fact that the set $\tilde W_N$ is convex and so the function $d(\cdot,\tilde W_N)$ is convex. Therefore, for each $k\in\N$,  there is  $N\geq k$ such that $d\left(\omega, \tilde W_N\right)<\delta$.  It follows that $d\left(\omega, \bigcup_{n \geq k}\tilde{W_n}\right)<\delta$. Since $\delta$ was arbitrary, we conclude that $\omega\in\overline{\bigcup_{n \geq k}\tilde{W_n}}$, for the given $k$. That is, $\omega\in\bigcap_{k \geq 1}\overline{\bigcup_{n \geq k}\tilde{W_n}}$ and so
\[W_e=\Conv\left(\E\right)\subset \bigcap_{k\geq 1}\overline{\bigcup_{n \geq k} \tilde{W}_n}.\]
Now,  since theorem \ref{BlocksThm} is valid for any decomposition, we have
\[
\Conv(\E)= W_e=\bigcap_{k \geq 1}\Conv\left(\overline{\bigcup_{n \geq k}\tilde{W_n}}\right) \supset \bigcap_{k \geq 1}\overline{\bigcup_{n \geq k}\tilde{W_n}}.
\]
The result then follows.\end{proof}
Next example elucidates that to find explicitly a decomposition as in proposition \ref{proposition} can be tricky. However, Theorem \ref{BlocksThm} makes possible to easily compute the essential numerical range.
\begin{example}
Let $\varphi: \N\to \Q\cap [0,2\pi]$ be any bijection and $\Theta: [0,2\pi[\to \bS_{\C}$ the map defined by $\Theta (\theta)=e^{i\theta}$. Note that $\Theta$ is continuous and bijective. Now, define the composite map \[
\phi=\Theta\circ\varphi: \N\to\bS_{\C}\,,\,\phi(n)=e^{i\varphi(n)}.
\]
Then, $\phi$ is also bijective. Define the bounded block diagonal operator $T=\diag\left\{\phi(n)\right\}=\bigoplus A_n$ with $A_n=\left[\phi(n)\right]$, for each $n\in\N$.
As seen in example \ref{ex_diag},
\[
W_e(T)=\Conv\left(\overline\lim \left\{\phi(n)\right\}_n\right).
\]
Since $\overline\lim \left\{\phi(n)\right\}_n=\bS_\C$ and $W_e(T)$ is convex, we conclude that $W_e(T)=\{x\in \C: \|x\|\leq 1\}$ is the closed unit disc.

%
%
\end{example}

\end{document}